\documentclass[11pt,letterpaper]{amsart}
\usepackage{amsmath,amssymb,txfonts}
\usepackage{amssymb}
\usepackage{amsxtra}
\usepackage{amsthm, color}
\usepackage{txfonts}
\usepackage{graphicx}
\usepackage{times}
\usepackage{citeref}
\usepackage{tikz}
\usepackage{pgfplots}
\usepackage{tikz-3dplot}
\numberwithin{equation}{section}
\usepackage{pb-diagram}
\usepackage{tkz-euclide}
\newcommand{\hgline}[2]{
	\pgfmathsetmacro{\thetaone}{#1}
	\pgfmathsetmacro{\thetatwo}{#2}
	\pgfmathsetmacro{\theta}{(\thetaone+\thetatwo)/2}
	\pgfmathsetmacro{\phi}{abs(\thetaone-\thetatwo)/2}
	\pgfmathsetmacro{\close}{less(abs(\phi-90),0.0001)}
	\ifdim \close pt = 11pt
	\draw[red] (\theta+180:1) -- (\theta:1);
	\else
	\pgfmathsetmacro{\R}{tan(\phi)}
	\pgfmathsetmacro{\distance}{sqrt(1+\R^2)}
	\draw[red] (\theta:\distance) circle (\R);
	\fi
}

\newtheorem{prop}{Proposition}[section]
\newtheorem{theorem}[prop]{Theorem}

\usepackage{tikz}
\usetikzlibrary{calc}

\makeatletter
\newcommand\smallO{
	\mathchoice
	{{\scriptstyle\mathcal{O}}}
	{{\scriptstyle\mathcal{O}}}
	{{\scriptscriptstyle\mathcal{O}}}
	{\scalebox{.7}{$\scriptscriptstyle\mathcal{O}$}}
}
\def\hyper@x#1,#2\relax{#1}
\def\hyper@y#1,#2\relax{#2}
\def\hyper@coords#1{#1}

\newif\ifhyper@vertical

\def\hyper@computer#1#2{%
	\edef\hyper@toscan{(#1)}
	\tikz@scan@one@point\hyper@coords\hyper@toscan
	\edef\hyper@sx{\the\pgf@x}
	\edef\hyper@sy{\the\pgf@y}
	\edef\hyper@toscan{(#2)}
	\tikz@scan@one@point\hyper@coords\hyper@toscan
	\edef\hyper@ex{\the\pgf@x}
	\edef\hyper@ey{\the\pgf@y}
	\pgfmathsetmacro{\hyper@mx}{(\hyper@ex + \hyper@sx)/2}
	\pgfmathsetmacro{\hyper@my}{(\hyper@ey + \hyper@sy)/2}
	\pgfmathsetmacro{\hyper@dx}{\hyper@ex - \hyper@sx}
	\pgfmathparse{\hyper@dx == 0 ? "\noexpand\hyper@verticaltrue" : "\noexpand\hyper@verticalfalse"}
	\pgfmathresult
	\ifhyper@vertical
	\edef\hyper@cmd{-- (\tikztotarget)}
	\else
	\pgfmathsetmacro{\hyper@dy}{\hyper@ey - \hyper@sy}
	\pgfmathsetmacro{\hyper@t}{\hyper@my/\hyper@dx}
	\pgfmathsetmacro{\hyper@cx}{\hyper@mx + \hyper@t * \hyper@dy}
	\pgfmathsetmacro{\hyper@radius}{veclen(\hyper@cx - \hyper@sx, \hyper@sy)}
	\pgfmathsetmacro{\hyper@sangle}{180 - atan2(\hyper@sy,\hyper@cx-\hyper@sx)}
	\pgfmathsetmacro{\hyper@eangle}{180 - atan2(\hyper@ey,\hyper@cx-\hyper@ex)}
	\edef\hyper@cmd{arc[radius=\hyper@radius pt, start angle=\hyper@sangle, end angle=\hyper@eangle]}
	\fi
}

\def\hyper@disc@computer#1#2{%
	\edef\hyper@toscan{(#1)}
	\tikz@scan@one@point\hyper@coords\hyper@toscan
	\edef\hyper@sx{\the\pgf@x}
	\edef\hyper@sy{\the\pgf@y}
	\edef\hyper@toscan{(#2)}
	\tikz@scan@one@point\hyper@coords\hyper@toscan
	\edef\hyper@ex{\the\pgf@x}
	\edef\hyper@ey{\the\pgf@y}
	\pgfmathsetmacro{\hyper@det}{\hyper@sx * \hyper@ey - \hyper@sy * \hyper@ex}
	\pgfmathparse{\hyper@det == 0 ? "\noexpand\hyper@verticaltrue" : "\noexpand\hyper@verticalfalse"}
	\pgfmathresult
	\ifhyper@vertical
	\edef\hyper@cmd{-- (\tikztotarget)}
	\else
	\pgfmathsetmacro{\hyper@mx}{(\hyper@ex + \hyper@sx)/2}
	\pgfmathsetmacro{\hyper@my}{(\hyper@ey + \hyper@sy)/2}
	\pgfmathsetmacro{\hyper@dx}{\hyper@ex - \hyper@sx}
	\pgfmathsetmacro{\hyper@dy}{\hyper@ey - \hyper@sy}
	\pgfmathsetmacro{\hyper@dradius}{\pgfkeysvalueof{/tikz/hyperbolic disc radius}}
	\pgfmathsetmacro{\hyper@t}{((\hyper@dradius)^2 - \hyper@sx * \hyper@ex - \hyper@sy * \hyper@ey)/(2 * (\hyper@sx * \hyper@ey - \hyper@sy * \hyper@ex))}
	\pgfmathsetmacro{\hyper@radius}{sqrt((\hyper@t)^2 + .25) * veclen(\hyper@dx,\hyper@dy)}
	\pgfmathsetmacro{\hyper@cx}{\hyper@mx + \hyper@t * \hyper@dy}
	\pgfmathsetmacro{\hyper@cy}{\hyper@my - \hyper@t * \hyper@dx}
	\pgfmathsetmacro{\hyper@sangle}{atan2(\hyper@sy-\hyper@cy,\hyper@sx - \hyper@cx)}
	\pgfmathsetmacro{\hyper@eangle}{atan2(\hyper@ey-\hyper@cy,\hyper@ex - \hyper@cx)}
	\pgfmathsetmacro{\hyper@eangle}{\hyper@eangle > \hyper@sangle + 180 ? \hyper@eangle - 360 : \hyper@eangle}
	\edef\hyper@cmd{arc[radius=\hyper@radius pt, start angle=\hyper@sangle, end angle=\hyper@eangle]}
	\fi
}

\def\hyper@plane@tangent#1#2{%
	\edef\hyper@toscan{(#1)}
	\tikz@scan@one@point\hyper@coords\hyper@toscan
	\edef\hyper@sx{\the\pgf@x}
	\edef\hyper@sy{\the\pgf@y}
	\edef\hyper@toscan{(#2)}
	\tikz@scan@one@point\hyper@coords\hyper@toscan
	\edef\hyper@ex{\the\pgf@x}
	\edef\hyper@ey{\the\pgf@y}
	\pgfmathsetmacro{\hyper@ex}{\hyper@ex - \hyper@sx}
	\pgfmathsetmacro{\hyper@ey}{\hyper@ey - \hyper@sy}
	\pgfmathparse{\hyper@ex == 0 ? "\noexpand\hyper@verticaltrue" : "\noexpand\hyper@verticalfalse"}
	\pgfmathresult
	\ifhyper@vertical
	\pgfmathsetmacro{\hyper@d}{\hyper@ey/1cm}
	\pgfmathsetmacro{\hyper@radius}{\hyper@sy * exp(\hyper@d) - \hyper@sy}
	\edef\hyper@cmd{-- ++(0,\hyper@radius pt)}
	\else
	\pgfmathsetmacro{\hyper@d}{\hyper@ex > 0 ? veclen(\hyper@ex,\hyper@ey) : -veclen(\hyper@ex,\hyper@ey)}
	\pgfmathsetmacro{\hyper@radius}{abs(\hyper@sy * \hyper@d / \hyper@ex)}
	\pgfmathsetmacro{\hyper@sangle}{90 + atan(\hyper@ey/\hyper@ex)}
	\pgfkeysgetvalue{/tikz/hyperbolic plane target angle}{\hyper@eangle}
	\ifx\hyper@eangle\pgfutil@empty
	\pgfmathsetmacro{\hyper@d}{\hyper@d/1cm}
	\pgfmathsetmacro{\hyper@ey}{\hyper@ey/1cm}
	\pgfmathsetmacro{\hyper@tanhd}{tanh(\hyper@d)}
	\pgfmathsetmacro{\hyper@eangle}{acos((\hyper@d * \hyper@tanhd - \hyper@ey)/(\hyper@d - \hyper@ey * \hyper@tanhd))}
	\fi
	\edef\hyper@cmd{arc[radius=\hyper@radius pt, start angle=\hyper@sangle, end angle=\hyper@eangle]}
	\fi
}

\tikzset{%
	hyperbolic disc radius/.initial={1cm},
	hyperbolic plane/.style={
		to path={
			\pgfextra{\hyper@computer\tikztostart\tikztotarget}
			\hyper@cmd
		}
	},
	hyperbolic plane tangent/.style={
		to path={
			\pgfextra{\hyper@plane@tangent\tikztostart\tikztotarget}
			\hyper@cmd
		}
	},
	hyperbolic disc/.style={
		to path={
			\pgfextra{\hyper@disc@computer\tikztostart\tikztotarget}
			\hyper@cmd
		}
	},
	hyperbolic plane target angle/.initial={},
}

\makeatother
\begin{document}
	
\title[Sharply estimating hyperbolic capacities]{Sharply estimating hyperbolic capacities}
\author{Xiaoshang Jin}
\address{School of Mathematics and Statistics, Huazhong University of Science and Technology, Wuhan, Hubei 430074, China}
	\email{jinxs@hust.edu.cn}
\author{Jie Xiao}
\address{Department of Mathematics \& Statistics,
	Memorial University, St. John's, NL A1C 5S7, Canada}
\email{jxiao@math.mun.ca}

\thanks{The first-named and second-named authors were supported by {NNSF of China
	\# 12201225} and NSERC of Canada \# 202979 respectively.}

\date{}


\maketitle


\begin{abstract} This paper is devoted to establishing
four types of sharp capacitary inequalities within the hyperbolic space as detailed in Theorems \ref{t21}-\ref{t31}-\ref{t41}-\ref{t51}.
\end{abstract}

\tableofcontents

\section{Hyperbolic capacities of general condensers}\label{s1} 
Suppose that $\mathbb H^{n+1}$ is the hyperbolic $(n+1)$-space with the Poincar\'e disk model 
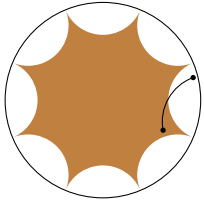
\begin{center}
\begin{tikzpicture}[hyperbolic disc radius=1.3cm,
every to/.style={hyperbolic disc}]
\draw (0,0) circle[radius=\pgfkeysvalueof{/tikz/hyperbolic disc radius}];
\pgfmathsetmacro{\hyperrad}{1/(2*sin(22.5))}
\fill[brown] (-112.5:\hyperrad) \foreach \k in {0,...,7} { to ++(\k * 45:1)};
\coordinate (b) at (1.2,.3);
\coordinate (a) at (.8,-.4);
\fill (a) circle[radius=1pt];
\fill (b) circle[radius=1pt];
\draw (a) to (b);
\end{tikzpicture}
\end{center}
which is equipped with the Riemannian metric 
$$g=(g_{ij})\ \ \text{where}\ \
g_{ij}(x)=\Bigg(\frac{2}{1-|x|^2}\Bigg)^2\delta_{ij}(x)
$$
living in the open unit ball $\mathbb B^{n+1}$ of the Euclidean space $\mathbb R^{n+1}$. If  
$$
d_g(x,y)=\rm{arcosh}\left(\frac{(1+|x|^2)(1+|y|^2)-4x\cdot y}{(1-|x|^2)(1-|y|^2)}\right)\ \ \forall\ \ x,y\in\mathbb B^{n+1}
$$
stands for the hyperbolic distance between $x\ \& \ y$, then 
$$
B_x(r)=\big\{y\in\mathbb H^{n+1}:\ d_g(x,y)<r\big\}
$$
represents $\mathbb B^{n+1}\ni x$-centered hyperbolic geodesic ball with radius $r>0$, and hence 
$$
\partial B_x(r)= \big\{y\in\mathbb H^{n+1}:\ d_g(x,y)=r\big\}
$$
is the $\mathbb B^{n+1}\ni x$-centered hyperbolic geodesic sphere.  

Moreover, for any Borel set $B\subseteq\mathbb B^{n+1}$ with its boundary $\partial B$ the hyperbolic voulme $\upsilon_g$ and the hyperbolic surface-area $\sigma_g$ are defined by
$$
\upsilon_g(B)=\int_B\left(\frac{2}{1-|x|^2}\right)^{n+1}\,d\upsilon(x)\ \ \text{and}\ \ \sigma_g(\partial B)=\int_{\partial B}\left(\frac{2}{1-|x|^2}\right)^{n}\,d\sigma(x),
$$
where $d\upsilon$ and $d\sigma$ denote the $(n+1)$-dimensional volume (Lebesgue measure) element and the $n$-dimensional surface-area (Hausdorff measure) element respectively. In particular, we have
\begin{equation}
\label{1.1}
\begin{cases}
\upsilon_g(B_0(r))=\sigma_{n}\int_0^r\sinh^n t\,dt;\\
\sigma_g\big(\partial B_0(r)\big)=\sigma_{n}\sinh^n r;\\
B_0(r)\equiv B(r)=\text{the origin-centered geodesic ball with radius}\ r;\\
\sigma_n=\text{the surface-area of}\ \mathbb B^{n+1}.
\end{cases}
\end{equation}

According to \cite{BDS}, a function $f\in L^1_{loc}(\mathbb H^{n+1})$ is of locally bounded variation amounts to the total variation of $f$ on  $U\Subset\mathbb H^{n+1}$ is finite - namely -
$$
\|Df\|(U)=\sup\Bigg\{\int_U f\text{div}\Phi\,d\upsilon_g:\ \Phi\in C_0^1(U;\mathbb R^{n+1})\ \&\ \sup_{x\in U}|\Phi(x)|\le 1\Bigg\}<\infty.
$$
Especially, if $f$ is of $C^1$-class, then
\begin{equation}
\label{1.2}
\|Df\|(U)=\int_U|\nabla f|\,d\upsilon_g\ \ \forall\ \ U\Subset\mathbb H^{n+1}.
\end{equation}

In the above and below, $\text{div}\Phi$ and $|\nabla f|$ stand for the divergence of $\Phi$ and the norm of the gradient $\nabla f$ of $f$  as with respect to the previously-defined metric $g$.
Our starting point then is to extend \eqref{1.2} through considering the following minimal $(0,\infty)\ni p$-energy -i.e.- the hyperbolic $p$-capacity:
\begin{equation}
\label{{1.3}}
{{\rm cap}_p(K,O)}=\inf\Bigg\{\int_{O}|\nabla f|^p\,d\upsilon_g:\ C_0^1(O)\ni f\ge 1_K\Bigg\}
\end{equation}
for any open set $O\subseteq\mathbb H^{n+1}$ and any compact subset $K$ of $O$ with $1_K$ being its characteristic function. Physically speaking,  $(K,O)$ exists as a general condenser of the capacitance ${\rm cap}_p(K,O)$ within  $\mathbb H^{n+1}$. However, ${\rm cap}_{0<p\le 1}(K,O)$
deserves special attention thanks to 
 \begin{equation}
\label{{1.4}}
\begin{cases}
{\rm cap}_{0<p<1}(K,O)=0;\\
{\rm cap}_1(K,O)=\inf \big\{\sigma_g(\partial\Omega):\ K\subseteq\text{smooth}\ \Omega\subseteq O\big\}.
\end{cases}
\end{equation}

As a matter of fact, upon not only choosing such a smooth domain $\Omega\Subset\mathbb H^{n+1}$ that $K\subset \Omega\subseteq O$ but also setting
$$
\begin{cases}
d_g(\Omega,\cdot)=\text{ the hyperbolic distance function of}\ \Omega;\\
\big\{x\in\mathbb H^{n+1}:\ d_g(\Omega, x)\leq R_0\big\}\subset O\ \ \text{for some}\ \ R_0>0;\\
f(x)=\max\Big\{1-\big(R_0^{-1}{d_g(\Omega,x)}\big)^{\rm m},0\Big\}\ \forall\ x\in O;\\
 {\rm m}\gg 1>p>0,
 \end{cases}
 $$
we get
\begin{align*}
{\rm cap}_p(K,O)&\leq \int_O|\nabla f|^p d\upsilon_g\\
&=\int_0^{R_0} {\rm m}^p\big({s}{R^{-1}_0}\big)^{({\rm m}-1)p}\sigma_g\Big(\big\{x\in\mathbb H^{n+1}:\ d_g(\Omega, x)=R_0\big\}\Big)\,ds\\
&\leq \max\limits_{s\in[0,R_0]}{\sigma}_g\Big(\big\{x\in\mathbb H^{n+1}:\ d_g(\Omega, x)= R_0\big\}\Big)\left(\frac{R_0{\rm m}^p}{1+({\rm {m}}-1)p}\right)\\
&\rightarrow 0\ \ {\rm as}\ \ {\rm m}\rightarrow\infty,
\end{align*}
thereby reaching \eqref{{1.4}}'s first formula. Needless to say, \eqref{{1.4}}'s second formula is a consequence of \eqref{1.2}\&\eqref{{1.3}}.

Our main interest in these hyperbolic capacities (cf. \cite{jin2022relative, li2023sharp, XAP} addressing some of the very basic aspects of ${\rm cap}_{1<p<\infty}(K,\mathbb H^{n+1})$) lies with not only the optimal geometric estimates for ${{\rm cap}_{1<p<\infty}(K,O)}$ (cf. Theorems \ref{t21}-\ref{t31} established within \S\ref{s2}-\S\ref{s3}) but also their naturally induced geometric \&/or physic objects (cf. Theorems \ref{t41}-\ref{t51} established within \S\ref{s4}-\S\ref{s5}). Of particular importance is the following formula as showed in \cite{grigor1999isoperimetric}:
\begin{equation}
\label{{1.5}}
{\rm cap}_{1<p<\infty}\big(\overline{B}(r),B(R)\big)=\left(\int_r^{R} (\sigma_n\sinh^n s)^{\frac{1}{1-p}}ds\right)^{1-p}\ \text{with}\ \overline{B}(r)=\text{the closure of}\ B(r),
\end{equation}
whose limiting case $p\to 1$ recovers \eqref{1.1}'s second formula due to \eqref{{1.4}}'s second formula. And yet, along with \eqref{{1.4}}'s first formula, artifically letting $p\to 0$ in \eqref{{1.5}} suggests the following formal defination
$$
{\rm cap}_{0}\big(\{0\},B(R)\big)\equiv\upsilon_g\big(B(R)\big)=\int_0^{R} (\sigma_n\sinh^n s)\, ds.
$$

Throughout we will write $X\sim Y$ when either $X\, \&\, Y$ or $X^{-1}\, \&\, Y^{-1}$ are equivalent infinitesimals.

\section{Sharp capacitary inequalities for smooth condensers}\label{s2}

If $K$ is a bounded closed domain with smooth boundary $\partial K$, then it is well-known that for $p\in (1,\infty)$ there exists a unique weak solution $u$ to the Dirichlet problem:
\begin{equation*}
\begin{cases}
 \text{div}\big(|\nabla u|^{p-2}\nabla u\big)=0\ \ &{\rm in}\  \ O\setminus K;\\
u=1 \ \  &{\rm on}\  \ \partial K;\\
u=0 \ \ &{\rm on}\ \ \partial O,
\end{cases}
\end{equation*}
such that
\begin{equation}\label{{2.1}}
{\rm cap}_{p}(K,O)=\int_{O\setminus K}|\nabla u|^{p} d\upsilon_g=\int_{\{u=t\}}|\nabla u|^{p-1}\, d\sigma_g
\end{equation}
holds for almost all $t\in[0,1].$ Accordingly, the function $u$ is said to be the hyperbolic $p$-capacitary potential of the condenser $(K,O)$.

The first goal of this paper is to establish a sharp isocapacitary inequality in $\mathbb H^{n+1}$. In order to provide a more accurate explanation of it, we define the hyperbolic volume radius ${\rm Rvol}(S)$ of a bounded set $S$ via solving the equation 
$$
\upsilon_g(S)=\upsilon_g\Big(B\big({\rm Rvol}(S)\big)\Big),
$$
thereby achieving a hyperbolic counterpart of \cite[(2.14)]{xiao2020geometrical} (based on \cite[Proposition 1.2]{hansen1994isoperimetric}).

\begin{theorem}\label{t21} Suppose that
	$$
	\begin{cases}
  0<t<1<p<\infty;\\
   (K,O)\ \text{is a condenser in $\mathbb{H}^{n+1}$ with the smooth boundary pair $\{\partial K,\partial O\}$};\\
 u\ \text{is the hyperbolic $p$-capacitary potential of $(K,O)$};\\
 D_t=\{x\in O\setminus K:\ u(x)\ge t\}\cup K;\\
 r_t={\rm Rvol}(D_t);\\
    r= {\rm Rvol}(K);\\
    	R={\rm Rvol}(O),
    \end{cases}
    $$
    Then

  \begin{itemize}
  	\item[\rm (i)]
  \begin{equation} \label{{2.2}}
 \exists\ t_0\in (0,1)\ \text{such that}
 \ \left(\frac{{\rm cap}_{p}(K,O)}{{\rm cap}_{p}\big(\overline{B}(r),B(R)\big)}\right)^{\frac{1}{p}}\geq \frac{{\sigma_g}\big(\partial D_{t_0}\big)}{{\sigma_g}\big(\partial B(r_{t_0})\big)}\geq 1.
 \end{equation}
 
 \item[\rm (ii)] 
 \begin{align}\label{{2.3}}
&K\, \&\, O\, \text{are hyperbolic geodesic balls similar to}\ \overline{B}(r)\, \&\, B(R)\\
&\ \ \Rightarrow\ {\rm cap}_{p}(K,O)={\rm cap}_{p}\big(\overline{B}(r), B(R)\big)\notag\\
&\ \ \Rightarrow\ K\, \&\, O\ \text{are hyperbolic geodesic balls}.\notag
 \end{align}
\end{itemize}
\end{theorem}
\begin{proof} (i)
	For any $t\in[0,1]$ let 
	$$\partial D_t=\{x\in O:u(x)=t\}$$ be the boundary of $D_t.$ Then \eqref{{2.1}} derives
	\begin{equation*}
	{\rm cap}_p(K,O)=\int_{\partial D_t}|\nabla u|^{p-1}d\sigma_g\ \ \text{for almost every}\ \ t\in [0,1].
	\end{equation*}
	Note that the co-area formula produces
	\begin{equation*}
	\upsilon_g(D_t)=\upsilon_g(K)+\int_t^1\int_{\partial D_s}|\nabla u|^{-1}d\sigma_g ds
	\end{equation*}
	So, the H\"older inequality further implies
	\begin{align*}
	\upsilon_g(\partial D_t)&\leq\left(\int_{\partial D_t} |\nabla u|^{p-1} d\sigma_g \right)^{\frac{1}{p}}\left(\int_{\partial D_t} |\nabla u|^{-1} d\sigma_g \right)^{1-\frac{1}{p}}\\
	&=\big({\rm cap}_p(K,O)\big)^{\frac{1}{p}}\left(-\frac{d}{dt}\upsilon_g(D_t)\right)^{1-\frac{1}{p}}.
	\end{align*}
	
Also, for $s\in (0,\infty)$ we may assume
	\begin{equation*}
\begin{cases}	
f(s)=\sigma_n\sinh^n s=\sigma_g\big(\partial B(s)\big);\\
F(s)=\int_0^s\sigma_n\sinh^n t\, dt=\upsilon_g(B(s)\big),
\end{cases}
	\end{equation*}
thereby considering the hyperbolic $p$-capacitary potential $v(x)$ of concentric ring $\big(\overline{B}(r),B(R)\big)$ centred at the origin:
	\begin{equation*}
	v(x)=\frac{\int_{d_g(0,x)}^R f(s)^{\frac{1}{1-p}}ds}{\int_{r}^R f(s)^{\frac{1}{1-p}}ds}
	\end{equation*}

If	
$$
\begin{cases} 
E_t=\{x\in B(R)\setminus \overline{B}(r):v(x)\geq t\}\cup \overline{B}(r);\\
\partial E_t=\{x\in B(R):v(x)=t\},
\end{cases}
$$
then
$$
\begin{cases} 
\upsilon_g(D_0)=\upsilon_g\big(B(R)\big)=\upsilon(E_0);\\
\upsilon_g(D_1)\geq \upsilon_g(K)=\upsilon_g\big(B(r)\big)=\upsilon_g(E_1),
\end{cases}
$$
and hence by the mean value theorem for derivative there exists a $t_0\in (0,1]$ such that
	\begin{equation*}
	\begin{cases}
	\upsilon_g(D_{t_0})-\upsilon_g(E_{t_0})=0;\\
	\frac{d}{dt}\big(\upsilon_g(D_t)-\upsilon_g(E_t)\big)\Big|_{t=t_0}\geq 0.
	\end{cases}
	\end{equation*}
	Since $|\nabla  v|$ is constant on $\partial E_t,$ we derive 
	\begin{equation*}
	\sigma_g(\partial E_t)=\Big({\rm cap}_p\big(\overline{B}(r),B(R)\big)\Big)^{\frac{1}{p}}\left(-\frac{d}{dt}\upsilon_g(E_t)\right)^{1-\frac{1}{p}}.
	\end{equation*}
	whence
	\begin{equation*}
	\left(\frac{{\rm cap}_{p}(K,O)}{{\rm cap}_{p}\big(\overline{B}(r),B(R)\big)}\right)^{\frac{1}{p}}\geq \frac{\sigma_g(\partial D_{t_0})}{\sigma_g(\partial E_{t_0})}.
	\end{equation*}
	Notice that $\partial E_{t_0}$ is a hyperbolic geodesic sphere with radius $${\rm Rvol} (E_{t_0})={\rm Rvol} (D_{t_0}).
	$$
	So we finish the proof of the first inequality of \eqref{{2.2}}.
	
	The second inequality of \eqref{{2.2}} follows immediately from the isoperimetric inequality in $\mathbb H^{n+1}$ that (cf. \cite[Section 6]{schmidt1943beweis} or \cite[(2.8)]{BDS})
	\begin{equation}\label{{2.4}}
	\sigma_g(\partial D_{t_0})\geq f\circ F^{-1}\big(\upsilon_g(D_{t_0})\big)=  f\circ F^{-1}\big(\upsilon_g(E_{t_0})\big)=\sigma_g(\partial E_{t_0}).
	\end{equation}
	
	(ii) It is enough to verify \eqref{{2.3}}'s second implication  `$\Rightarrow$' . Now, suppose
	\begin{equation}
	\label{{2.5}}
	{\rm cap}_{p}(K,O)={\rm cap}_{p}\big(\overline{B}(r),B(R)\big).
	\end{equation} 
	Then $D_{t_0}$ is a hyperbolic geodesic ball. Next, we will show that $K$ \& $O$ are two hyperbolic geodesic balls.
	
	To do so, recall Maz’ya's isocapacitary inequality in \cite[(2.2.8)]{maz2013sobolev}:
	\begin{equation}\label{{2.6}}
	{\rm cap}_p(K,O)\geq \left(\int_{\upsilon_g(K)}^{\upsilon_g(O)}\big(I(s)\big)^{\frac{p}{1-p}}\,ds\right)^{1-p},
	\end{equation}
	where $I(s)$ is not only an isoperimetric function but also equal to the infimum of $\sigma_g(\partial\Omega)$ over all precompact open sets $\Omega$ with $$
	\begin{cases}
	K\subset\Omega\subseteq O;\\
	\upsilon_g(\Omega)\geq s.
	\end{cases}
	$$ 
	So, if the equality of \eqref{{2.6}} holds, then
	$$I\big(\upsilon_g(D_t)\big)=\sigma_g(\partial D_t)\ \ \text{for almost every}\ \ t\in(0,1).
	$$
	Notice that within the hyperbolic space $\mathbb H^{n+1}$ there is
	$$I(s)=f\circ F^{-1}(s)\ \ \forall\ \ s\in (0,\infty).
	$$ 
	Thus
	\begin{align}\label{{2.7}}
	{\rm cap}_p(K,O)&\geq \left(\int_{\upsilon_g(K)}^{\upsilon(O)}\big(f\circ F^{-1}(s)\big)^{\frac{p}{1-p}}\,ds\right)^{1-p}
	\\ &=\left(\int_{F^{-1}(\upsilon_g(K))}^{F^{-1}(\upsilon_g(O)}\big(f(\theta)\big)^{\frac{p}{1-p}}F'(\theta)\,d\theta\right)^{1-p}\notag
	\\ &=\left(\int_{r}^{R}\big(f(\theta)\big)^{\frac{1}{1-p}}\,d\theta\right)^{1-p}\notag
	\\ &={\rm cap}_{p}\big(\overline{B}(r),B(R)\big).\notag
	\end{align}
	Accordingly, a combination of \eqref{{2.6}} \& \eqref{{2.5}} implies \eqref{{2.6}}'s equality case which in turn derives that $D_t$ is a hyperbolic geodesic ball for almost every $t\in(0,1).$ Therefore,
	$O=D_0$ \& $K=D_1$ are hyperbolic geodesic balls respectively.
	
\end{proof}

\section{Sharp capacitary inequalities for punctured condensers}\label{s3}

If $O\supset K\downarrow \{x\}$, then for $p>1$ there exists a function $u_x(\cdot)$ satisfying 
$$
\begin{cases}
\text{div}\big(|\nabla u_x|^{p-2}\nabla u_x\big)= 0 &{\rm in}\ \ O\setminus\{x\};\\
u_x(x)=1;\\
u_x=0 &{\rm on}\ \ \partial O,
\end{cases}
$$
and hence $u_x(\cdot)$ is called the hyperbolic $p$-capacitary potential of the puctured condenser $(\{x\},O)$. Consequently, we define the hyperbolic $p$-Green function $G_p(x,\cdot)$ of $O$ with singularity at $x$ by solving 
$$
\begin{cases} - \text{div}\big(|\nabla G_p(x,\cdot)|^{p-2}\nabla  G_p(x,\cdot)\big) =\delta(x,\cdot)=\text{the Dirac mass at}\ x\ {\rm in}\ \ O;\\
G_p(x,\cdot)\big|_{\partial O}=0.
\end{cases}
$$

The second goal of this paper is to establish the following sharp capacitary inequalities for $(\{x\},O)$.

\begin{theorem}\label{t31} Let
	\begin{equation*}
	\begin{cases}
	1<p<\infty;\\
	x\in O=\text{a bounded subdomain of}\ \mathbb{H}^{n+1}\ \text{with the smooth boundary}\ \partial O;\\
	\text{$u_{x}$=\ the hyperbolic $p$-capacitary potential of $(\{x\},O)$};\\
	s_p(x,O)=\lim\limits_{r\rightarrow 0}{\rm cap}_p\big(\overline{B}_x(r),O\big);\\
	\kappa_p(r)=\int_r^{\infty}(\sigma_n\sinh ^ns)^{\frac{1}{1-p}}\, ds;\\
	(\kappa_p)^{-1}=\text{the inverse of}\ \kappa_p;\\
	 \tau_p(x,O)=\lim\limits_{r\rightarrow 0}\Big(\kappa_p(r)-\big({\rm cap}_p\big(\overline{B}_x(r),O\big)\big)^{\frac{1}{1-p}}\Big);\\
	\rho_p(x,O)=(\kappa_p)^{-1}\circ\tau_p(x,O)=\text{the hyperbolic $p$-harmonic radius of $O$ at $x$}.
	\end{cases}
	\end{equation*}
	
	\begin{itemize}
		\item[\rm (i)] If 
		$$
		\begin{cases} 
		p>n+1;\\
		U=\big\{y\in O:\ u_{x}(y)>a\in (0,1)\big\};\\
		R={{\rm Rvol}(O)};\\
		r={{\rm Rvol}(U)},
		\end{cases}
		$$
		then 
		\begin{equation}\label{{3.1}}
		\frac{s_p(x,O)}{s_p(x,B_x(R))}-1\geq(p-1)\left(\frac{s_p\big(x,B_x(R)\big)}{s_p(x,B_{x}(r))}\right)^{\frac{1}{p-1}}
		\left(1-\Bigg(\frac{s_p\big(x,B_x(r)\big)}{s_p(x,U)}\Bigg)^{\frac{1}{p-1}}\right),
		\end{equation}
		holds with equality when $O=B_{x}(R)\  \&\  U=B_{x}(r)$.
		
		\item[\rm (ii)] If 
		$$
		\begin{cases} 
		p\le n+1;\\
		O=\text{a convex subdomain of}\ \mathbb H^{n+1};\\
		{\rm d}={\rm diam}(O)=\text{the diameter of}\ O,
		\end{cases}
		$$ 
		then
		\begin{equation}
		\label{{3.2}}
		\text{$\exists\ x_0\in O$\ such that\
			$\rho_p(x_0,O)\leq {\rm arccosh}\Bigg(\frac{\cosh \rm d}{\cosh\frac{\rm d}{2}}\Bigg)<\frac{\rm d}{2}+\ln 2$}
		\end{equation}
		holds with equality when $O$ is a hyperbolic $x_0$-centered geodesic ball.
	\end{itemize}
	\end{theorem}
\begin{proof}
	(i) Firstly, note that \eqref{{1.5}} yields
	\begin{equation}
	\label{{3.3}}
	\lim\limits_{r\rightarrow 0}{\rm cap}_{p}(\overline{B}_x(r),B_x(R))=0\Leftrightarrow p\in (1,n+1].
	\end{equation}
	So, as an application of
	regional monotonicity of ${\rm cap}_{p},$ we derive that if $p\in (1,n+1]$ then
	$$
	\lim\limits_{r\rightarrow 0}{\rm cap}_{p}\big(\overline{B}_x(r),O\big)=0
	$$
	for any bounded open set $O$ and $x\in O$, thereby finding that if $p>n+1.$ then for any $x\in O\subseteq\mathbb{H}^{n+1}$ and small $r>0,$ we have that
	\begin{equation}\label{{3.4}}
	{\rm cap}_p\big(\overline{B}_x(r),O\big)\geq\left(\int_0^{\infty}(\sigma_n\sinh^ns)^{\frac{1}{1-p}}ds\right)^{1-p}=\text{a constant depending only on}\ p
	\end{equation}
	and is an increasing function of $r$, and hence there is the formula
		$$s_p(x,O)=\lim\limits_{r\rightarrow 0}{\rm cap}_p\big(\overline{B}_x(r),O\big)\ \ \forall\ \  x\in O.
		$$
	
Secondly, suppose that $u_r$ is the hyperbolic $p$-capacitary potential of $\big(\overline{B}_x(r),O\big)$ for small $r>0.$ By the comparison principle, we have  $$u_{r_1}(y)\leq u_{r_2}(y)\ \ \forall\ \ r_1<r_2\ \ \&\ \ y\in O\setminus \overline{B}_x(r_2),
	$$	
	thereby reaching the limit
	$$
	\lim\limits_{r\rightarrow 0}u_r(y)=u_x(y)\ \ \forall\ \ y\in O\setminus\{x\}.
	$$
	Although not needed, it can be proved that the above convergence is in $C^{1,\alpha}$ norm on every compact set of $O\setminus\{x\}.$ However, what we want is \cite[Proposition 4.2]{janfalk1996behaviour}'s hyperbolic extension - if $x\in O\subseteq \mathbb{H}^{n+1}$ then
		\begin{equation}
		\label{{3.5}}
		s_p(x,O)=G_p(x,x)^{1-p}\ \ \text{under}\ \ p>n+1.
		\end{equation}
	
	In order to check the implication, using the formula 
		$$u_x(y)=\frac{G_p(x,y)}{G_p(x,x)}\ \ \forall\ \ y\in O,
		$$
		we obtain
		\begin{equation*}
		\begin{aligned}
		s_p(x,O)&=\int_\Omega|\nabla  u_x(y)|^p d\upsilon_g(y)
		\\ &=G_p(x,x)^{-p}\int_\Omega|\nabla  G_p(x,y)|^p d\upsilon_g(y)
		\\ &=-G_p(x,x)^{-p}\int_\Omega G_p(x,y)\text{div}\big(|\nabla  G_p(x,y)|^{p-2}\nabla  G_p(x,y)\big)\,d\upsilon_g(y)
		\\ &=G_p(x,x)^{1-p},
		\end{aligned}
		\end{equation*}
		as desired in \eqref{{3.5}}

	Thirdly, we will extend  \cite[Lemma 1.3]{hansen1994isoperimetric} to the assertion that if
	 $(K,O)$ is a condenser in $\mathbb{H}^{n+1}$ and $u$ is its hyperbolic $1<p$-capacitary potential as well as 
	 $$
	 U=\big\{x\in O\setminus K:\ u(x)>a\big\}\ \ \forall\ \ a\in (0,1)
	 $$
	 then  
		\begin{equation}\label{{3.6}}
		{\rm cap}_{p}(K,O)\geq {\rm cap}_{p}\big(\overline{B}(r_K),B(r_O)\big)+\frac{(p-1){\rm cap}_{p}\big(\overline{B}(r_K),B(r_O)\big)^{\frac{p}{p-1}}}{\left({\rm cap}_{p}(\overline{B}(r_K),B(r_U))^{\frac{1}{1-p}}-{\rm cap}_{p}(K,U)^{\frac{1}{1-p}}\right)^{-1}}.
\end{equation}
		where $r_K,r_U,r_O$ are the hyperbolic volume radii of $K,U, O$ respectively.

In order to verify \eqref{{3.6}}, let 
$$
u_1(\cdot)=(1-a)^{-1}{\max\{u(\cdot),0\}}\ \ \text{and}\ \ 
u_2(\cdot)=a^{-1}{\min\{u(\cdot),a\}}
$$ 
be the hyperbolic $p$-capacitary potentials of
		$(K,U)$ and $(\overline{U},O)$ respectively. Then
		$$u=(1-a)u_1+au_2.$$
		Upon setting
		$$
		\begin{cases} C= {\rm cap}_{p}(K,O)\geq c={\rm cap}_{p}\big(\overline{B}(r_K),B(r_O)\big);\\
		C_1={\rm cap}_{p}(K,U)\geq  c_1={\rm cap}_{p}\big(\overline{B}(r_K),B(r_U)\big);\\
		C_2={\rm cap}_{p}(\overline{U},O)\geq c_2={\rm cap}_{p}\big(\overline{B}(r_U),B(r_O)\big),
		\end{cases}
		$$
		we find not only
		$$c=(c_1^{\frac{1}{1-p}}+c_2^{\frac{1}{1-p}})^{1-p}$$ 
		but also
		\begin{equation*}
		\begin{aligned}
		C&=\int_{O\setminus K}|\nabla u|^p\,d\upsilon_g\\
		&=(1-a)^p \int_{U\setminus K}|\nabla  u_1|^p\,d\upsilon_g+a^p\int_{O\setminus U}|\nabla  u_2|^p\,d\upsilon_g
		\\ &=(1-a)^p C_1+a^p C_2
		\\ &\geq \left(1-\frac{C_1^{\frac{1}{p-1}}}{C_1^{\frac{1}{p-1}}+C_2^{\frac{1}{p-1}}}\right)^p C_1+\left(\frac{C_1^{\frac{1}{p-1}}}{C_1^{\frac{1}{p-1}}+C_2^{\frac{1}{p-1}}}\right)^p C_2
		\\&=(C_1^{\frac{1}{1-p}}+C_2^{\frac{1}{1-p}})^{1-p}
		\end{aligned}
		\end{equation*}
		As a consequence of the mean value theorem, there is $\xi\in (0,1)$ such that
		\begin{equation*}
		\begin{aligned}
		C-c &\geq\big(C_1^{\frac{1}{1-p}}+C_2^{\frac{1}{1-p}}\big)^{1-p}-\big(c_1^{\frac{1}{1-p}}+c_2^{\frac{1}{1-p}}\big)^{1-p}
		\\&=(1-p)\xi^{-p}\Big(C_1^{\frac{1}{1-p}}+C_2^{\frac{1}{1-p}}-c_1^{\frac{1}{1-p}}-c_2^{\frac{1}{1-p}}\Big)
		\\&\geq (p-1) \big(c_1^{\frac{1}{1-p}}+c_2^{\frac{1}{1-p}}\big)^{-p}\Big(\big(c_1^{\frac{1}{1-p}}-C_1^{\frac{1}{1-p}}\big)+\big(c_2^{\frac{1}{1-p}}-C_2^{\frac{1}{1-p}}\big)\Big)
		\\&\geq(p-1) c^{\frac{p}{p-1}}\Big(c_1^{\frac{1}{1-p}}-C_1^{\frac{1}{1-p}}\Big)
		\end{aligned}
		\end{equation*}
	
	Finally, choosing $K=\overline{B}_x(r)$ and letting $r\to 0$ in \eqref{{3.6}} we get \eqref{{3.1}} whose equality case can be validated easily.
	
	(ii) Firstly, just like \cite{bandle1996harmonic, cardaliaguet2002strict, wang2007n, xiao2020geometrical} \eqref{{3.4}} reveals that $\tau_p(x,O)$ \& $\rho_p(x,O)$ are less significant whenever $p>n+1$. Thus, we can readily find that if $r\to 0$ then
	$$
	\kappa_p(r)\sim\begin{cases} \sigma_n^{\frac{1}{1-p}}\left(\frac{p-1}{n+1-p}\right) r^{\frac{n+1-p}{1-p}}\ \ &\text{under}\ \ p\in (1,n+1);\\
	\sigma_n^{\frac{1}{1-p}}\ln r\ \ &\text{under}\ \ p=n+1.
	\end{cases}
	$$
	
	Secondly, we prove that if $O$ is a bounded domain in $\mathbb{H}^{n+1}$ and $B_x(r)$ is the hyperbolic geodesic ball centered at $x\in O$ with small radius $r>0$ then the function
		$$\kappa_p(r)-\left({\rm cap}_p\big(\overline{B}_x(r),O\big)\right)^{\frac{1}{1-p}}$$
		is not only positive but also increasing in $r$. 
	
	Indeed, the positivity is obvious thanks to
	$$
	\eqref{{1.5}}\Rightarrow\kappa_p(r)=\lim_{R\to\infty}\left({\rm cap}_p\big(\overline{B}_x(r), B_x(R)\big)\right)^{\frac{1}{1-p}}.
	$$ 
	Now, for 
	$$r_2>r_1>0\ \ \&\ \ B_x(r_2)\subseteq O,
	$$ 
	the subadditivity of the conformal modulus in \cite[Lemma 2.1]{flucher1999variational} implies not only
		$$
		\left({\rm cap}_p\big(\overline{B}_x(r_1),O\big)\right)^{\frac{1}{1-p}}\geq \left({\rm cap}_p\big(\overline{B}_x(r_1),B_x(r_2)\big)\right)^{\frac{1}{1-p}}+\left({\rm cap}_p\big(\overline{B}_x(r_2),O\big)\right)^{\frac{1}{1-p}},
		$$
		but also the desired conclusion due to
		$$\left({\rm cap}_p\big(\overline{B}_x(r_1),B_x(r_2)\big)\right)^{\frac{1}{1-p}}=\kappa_p(r_1)-\kappa_p(r_2).
		$$
	
	Thirdly, in accordance with \cite[p. 8]{flucher1999variational},
		for 
		$$
		\begin{cases} 
		p\in (1,n+1];\\
		x\in \mathbb{H}^{n+1};\\
		r>0,
		\end{cases}
		$$ 
		there is
		$$
		\begin{cases} \tau_p(x,B_x(r))=\kappa_p(r);\\
		\rho_p\big(x,B_x(r)\big)=r.
		\end{cases} 
		$$
Now, suppose that not only $O\subseteq\mathbb H^{n+1}$ is a bounded convex domain with its diameter 
$${\rm d}=d_g(y,z)\ \ \text{for some}\ \ y,z\in\partial O
$$ but also $x_0\in O$ (due to $O$'s convexity) is the center of the hyperbolic geodesic line $\overline{yz}.$ Moreover, for any $w\in O$ let 
		$$
		\begin{cases} 
		a=d_g(z,w);\\
		b=d_g(y,w);\\
		r=d_g(x_0,w).
		\end{cases}
		$$
		Then 
		$$\max\{a,b\}\leq {\rm d}.
		$$ 
		According to the law of cosine, there holds
		$$
		\cos\angle zyw=\frac{\cosh b\cosh {\rm d}-\cosh a}{\sinh b\sinh {\rm d}}=\frac{\cosh b\cosh \frac{{\rm d}}{2}-\cosh r}{\sinh b\sinh \frac{{\rm d}}{2}}.
		$$
		Furthermore, we gain
		\begin{equation*}
		\begin{aligned}
		\cosh r&=\cosh b\cosh \frac{{\rm d}}{2}-\left(\frac{\sinh\frac{{\rm d}}{2}}{\sinh {\rm d}}\right)(\cosh b\cosh {\rm d}-\cosh a)
		\\ &=\left(\frac{\sinh\frac{{\rm d}}{2}}{\sinh {\rm d}}\right)(\cosh b+\cosh a)\\
		&\leq \frac{\cosh {\rm d}}{\cosh\frac{{\rm d}}{2}}.
		\end{aligned}
		\end{equation*}
		Consequently, \eqref{{3.2}}'s first inequality follows from
		 $$
		 \begin{cases} d_g(x_0,w)\leq r({\rm d})={\rm arccosh}\Big(\frac{\cosh {\rm d}}{\cosh\frac{{\rm d}}{2}}\Big);\\
		 O\subseteq B_{x_0}(r({\rm d}));\\
	\rho_p(x_0,O)\leq \rho_p\big(x_0,B_{x_0}(r({\rm d}))\big)=r({\rm d}).
	\end{cases}
	$$
	Of course, the equality case of the last inequality occurs at $O=B_{x_0}(r({\rm d}))$.
		
		Finally, we obtain 
		\begin{equation}
		\label{{3.7}}
		r({\rm d})<\frac{{\rm d}}{2}+\ln 2\Leftrightarrow \frac{\cosh {\rm d}}{\cosh\frac{{\rm d}}{2}}<\cosh\Bigg(\frac{{\rm d}}{2}+\ln 2\Bigg).
		\end{equation} 
		Obviously, the right hand of \eqref{{3.7}} is true for any ${\rm d}>0$ - in fact - a redundant computation indicates that $r({\rm d})-\frac{{\rm d}}{2}$ is a monotonically increasing function of ${\rm d}$ and
		tends to $\ln 2$ as ${\rm d}\rightarrow\infty.$ This verifies \eqref{{3.7}}.
	
\end{proof}

\section{Sharp capacitary inequalities for smooth compact domains}\label{s4} 

Due to the monotonicity of 
$$
O\mapsto {\rm cap}_p(K,O)
$$
we can define
$$
{\rm cap}_p(\Omega)=\lim\limits_{R\rightarrow\infty}{\rm cap}_p\big(\Omega,B(R)\big)$$
 as the capacity of the compact (bounded \& closed) domain $\Omega\subseteq\mathbb{H}^{n+1},$
thereby getting
\begin{equation*}
{\rm cap}_p\big(\overline{B}(r)\big)=\left(\int_r^{\infty} (\sigma_n\sinh^n s)^{\frac{1}{1-p}}ds\right)^{1-p}.
\end{equation*}

According to \cite[Theorem 4.1]{fogagnolo2022minimising}, if the $1<p$-Green’s function of a $p$-nonparabolic Riemannian manifold $(\mathbb M,g)$ vanishes
at infinity, then for any bounded domain $\Omega\subseteq \mathbb M$ with smooth boundary, there exists a unique weak solution to the equation
$$
\begin{cases}
\text{div}(|\nabla u|^{p-2}\nabla  u)=0\ &{\rm in}\  \ \mathbb M\setminus \Omega\\ 
u=1 \ &{\rm on}\  \ \partial \Omega\\ 
u(x)\rightarrow 0 \ &{\rm as}\ \ d_g(\Omega,x)\rightarrow\infty,
\end{cases}
$$
and hence $u$ is called the $p$-capacitary potential of $\Omega$. Consequently, the hyperbolic $p$-capacitary potential always exists in $\mathbb  H^{n+1}$ due to that not only the curvature of $\mathbb H^{n+1}$ is negative but also $\mathbb H^{n+1}$ is $1<p$-nonparabolic with its $p$-Green function 
$$
G(x,y)=\int_{d_g(x,y)}^{\infty}(\sigma_n\sinh^n s)^{\frac{1}{1-p}}\,ds\to 0\ \ \text{as}\ \ d_g(x,y)\rightarrow \infty.
$$

The thrid goal of this paper is to sharply control ${\rm cap}_{p}(\Omega)$ from below and above.
\begin{theorem}\label{t41}
 Let 
 $$
 \begin{cases} 
 (p,t)\in (1,\infty)\times (0,1];\\
 \text{$\Omega$ be a compact domain with smooth boundary $\partial\Omega$ in $\mathbb{H}^{n+1}$};\\
 r={\rm Rvol}(\Omega);\\
 \text{$u$ be the hyperbolic $p$-capacitary potential of $\Omega$};\\
 D_{t}=\{x\in \mathbb{H}^{n+1}\setminus \Omega:\ u(x)\geq t\}\cup \Omega;\\
 r_t={\rm Rvol}(D_{t}).
 \end{cases}
 $$
 \begin{itemize}
 	\item [\rm (i)] Either
 \begin{equation}\label{{4.1}}
\text{$\exists\ t_0\in (0,1]$ such that $\left(\frac{{\rm cap}_{p}(\Omega)}{{\rm cap}_{p}\big(\overline{B}(r)\big)}\right)^{\frac{1}{p}}\geq \frac{\sigma_g\big(\partial D_{t_0}\big)}{\sigma_g\big(\partial B(r_{t_0})\big)}\geq 1$},
  \end{equation}
 or 
 \begin{equation}\label{{4.2}}
 \text{$\exists$\ a sequence\ $t_k\to 0$ such that}\  
\left(\frac{{\rm cap}_{p}(\Omega)}{{\rm cap}_{p}\big(\overline{B}(r)\big)}\right)^{\frac{1}{p}}=
\lim\limits_{k\rightarrow\infty}
\left(\frac{\sigma_g\big(\partial D_{t_k}\big)}{\sigma_g\big(\partial B(r_{t_k})\big)}\right)\geq 1.
  \end{equation}
Moreover,  
\begin{equation}
\label{{4.3}} 
\text{${\rm cap}_{p}(\Omega)={\rm cap}_{p}\big(\overline{B}(r)\big)\Leftrightarrow \Omega$ is a hyperbolic geodesic ball}.
\end{equation}

\item[\rm (ii)] Not only
  \begin{equation}\label{{4.4}}
  \text{$\bar{g}=u^{\frac{2(p-1)}{n}}g$\ can be extended continuously to\ $\overline{\mathbb{H}}^{n+1}=\mathbb{H}^{n+1}\cup \mathbb{S}^n$}
  \end{equation}
  but also
	\begin{equation}\label{{4.5}}
	{\rm cap}_p(\Omega)\leq \bigg(\frac{n}{p-1}\bigg)^{p-1}{\rm Vol}(\mathbb{S}^n,\bar{g}|_{\mathbb{S}^n})=\bigg(\frac{n}{p-1}\bigg)^{p-1}\sigma_{\bar{g}|_{\mathbb{S}^n}}(\mathbb{S}^n)
	\end{equation}
holds with inequality becoming equality when $\Omega$ is a hyperbolic geodesic ball.

\end{itemize}
\end{theorem}
\begin{proof}
	(i) Suppose that \eqref{{4.1}} does not hold -i.e.- there is
	\begin{equation*}
	\left(\frac{{\rm cap}_{p}(\Omega)}{{\rm cap}_{p}\big(\overline{B}(r)\big)}\right)^{\frac{1}{p}}<
	\frac{\sigma_g\big(\partial D_{t}\big)}{\sigma_g\big(\partial B(r_{t})\big)}\ \ \forall\ \ t\in (0,1].
	\end{equation*}
	Then
	\begin{equation}\label{{4.6}}
	\left(\frac{{\rm cap}_{p}(\Omega)}{{\rm cap}_{p}\big(\overline{B}(r)\big)}\right)^{\frac{1}{p}}\le\liminf_{t\to 0}
	\left(\frac{\sigma_g\big(\partial D_{t}\big)}{\sigma_g\big(\partial B(r_{t})\big)}\right).
	\end{equation}
	For any $a\in (0,1),$ we choose $$u_a(\cdot)=\max\Bigg\{\frac{u(\cdot)-a}{1-a},0\Bigg\}$$ as the hyperbolic $p$-capacitary potential of $\big(\Omega,(D_a)^\circ\big)$ where $(D_a)^\circ$ is the interior of $D_a$. Similarly, for any $\theta\in (0,1]$ we define  
	$$
D_\theta^a=\{x:u_a(x)\geq \theta\}\cup \Omega=\{x:u\geq (1-a)\theta+a\}\cup \Omega=D_{(1-a)\theta+a}.
	$$
	
	By Theorem \ref{t21}, there exists $\theta_0\in (0,1]$ (depending on $a$) such that
	\begin{equation*}
	\left(\frac{{\rm cap}_{p}(\Omega,(D_a)^\circ)}{{\rm cap}_{p}\big(\overline{B}(r), B(r_a)\big)}\right)^{\frac{1}{p}}\geq\frac{\sigma_g\big(\partial D^a_{\theta_0}\big)}{\sigma_g\Big(\partial B\big({\rm Rvol}(D^a_{\theta_0})\big)\Big)}=\frac{\sigma_g\big(\partial D_{\theta_a}\big)}{\sigma_g\big(\partial B(r_{\theta_a})\big)}\ge 1,
	\end{equation*}
	where 
	$$
	\theta_a=(1-a)\theta_0+a\in (0,1].
	$$ 
	Now, letting $a\rightarrow 0$ gives not only $\theta_a\to 0$ but also
	\begin{equation*}
	\left(\frac{{\rm cap}_{p}(\Omega)}{{\rm cap}_{p}\big(\overline{B}(r)\big)}\right)^{\frac{1}{p}}= \lim\limits_{a\rightarrow 0}\left(\frac{{\rm cap}_{p}(\Omega,(D_a)^\circ)}{{\rm cap}_{p}\big(\overline{B}(r),B(r_a)\big)}\right)^{\frac{1}{p}}\geq\limsup\limits_{a\rightarrow 0}\left(\frac{\sigma_g\big(\partial D_{t_a}\big)}{\sigma_g\big(\partial B(r_{\theta_a})\big)}\right).
	\end{equation*}
	This last estimation, along with \eqref{{4.6}}, produces a sequence $t_k=\theta_{a_k}\to 0$ such that
	\begin{equation*}
	\left(\frac{{\rm cap}_{p}(\Omega)}{{\rm cap}_{p}\big(\overline{B}(r)\big)}\right)^{\frac{1}{p}}=\lim\limits_{k\rightarrow \infty}\left(\frac{\sigma_g(\partial D_{t_{k}})}{\sigma_g\big(\partial B(r_{t_{k}})\big)}\right)\geq 1,
	\end{equation*}
	as desired in \eqref{{4.2}}.
	
	Obviously, \eqref{{4.3}} can be verified via a slight modification of the argument for \eqref{{2.3}}.
	
	(ii) First of all, let us point out that one of the major differences between hyperbolic space and Euclidean space is that the hyperbolic space is conformally compact due to its negative curvature. 
	
	Next, since the conformal infinity of $\mathbb{H}^{n+1}$ is actually the standard conformal sphere $\big(\mathbb{S}^n,[g_{\mathbb{S}}]\big)$, we are about to prove such an assertion that
if $\Omega\subseteq\mathbb{H}^{n+1}$ is a bounded smooth domain and $o\in \Omega$ is an interior point as well as $u$ is the hyperbolic $p$-capacitary potential
		of $\Omega$ for some $p>1$,  then for any unit vector $\xi\in T_o\mathbb{H}^{n+1}$ the limit
		\begin{equation}
		\label{{4.7}}
		\hat{u}(\xi)=\lim\limits_{t\rightarrow\infty}e^t\big(u(\exp_o(t\xi))\big)^{\frac{p-1}{n}}
		\end{equation}
not ony exists but also is a positive H\"older continuous function on $\mathbb{S}^n.$
	
	Indeed, if
		\begin{equation*}
		v(x)=G(o,x)=\int_{d_g(o,x)}^{\infty}(\sigma_n\sinh^n s)^{\frac{1}{1-p}}\,ds,
		\end{equation*}
		then
		$$
		\text{div}\big(|\nabla v(x)|^{p-2}\nabla v(x)\big)=0\ \ \forall\ \ x\neq o.
		$$
		For sufficiently large $R>0,$ let $u_R$ be the hyperbolic $p$-capacitary potential of the condenser $(\Omega, B_o(R)).$ Then up to a subsequence, there holds (cf. \cite[Appendix A]{fogagnolo2022minimising})
		$$
		u_R\rightarrow u\ \ \text{as}\ \ R\rightarrow\infty.
		$$ 
		On the other hand, the strong comparison principle
		for $u_R\ \&\ v$ indicates that $u_R\leq C_1 v$ for some $C_1>0$ independent of $R.$ As a consequence, we get 
		$$\frac{u}{v}\leq C_1\ \ \text{in}\ \ 
		\mathbb{H}^{n+1}\setminus\Omega.
		$$ 
		Similary, we can also obtain 
		$$
		\frac{u}{v}\geq C_2>0\ \ \text{in}\ \ 
		\mathbb{H}^{n+1}\setminus\Omega.
		$$ 
		Now that 
		$$
		0<C_2\leq \frac{u}{v}\leq C_1\ \ \text{in}\ \ \mathbb{H}^{n+1}\setminus\Omega.
		$$ 
		Thus a demonstration similar to the argument for \cite[Proposition 2.6]{schoen1994lectures} indicates that $\frac{u}{v}$ can be continuously extended to $\partial\mathbb{H}^{n+1}=\mathbb{S}^n$ and is not only H\"older continuous but also
		great than $C_2>0.$ Nevertheless, the existence of \eqref{{4.7}} follows from the fact that there is a constant $C_{n,p}>0$ such that
		$$v\big(\exp_o(t\xi)\big)=\int_{t}^{\infty}(\sigma_n\sinh^n s)^{\frac{1}{1-p}}ds\sim C_{n,p} e^{\frac{n t}{1-p}} \ \ \text{as}\ \ t\rightarrow\infty.$$
		
Last of all, as an application of \eqref{{4.7}}'s existence and continuity, we can utilize $u$ to construct a defining function. More concretely, let 
$$
g=dr^2+\sinh^2r g_\mathbb{S}\ \ \text{with}\ \ r=d_g(o,\cdot)
$$ be the hyperbolic metric. Consider the defining function $\rho=e^{-r}$ and the conformal compactification
	$$g_0=\rho^2 g=d\rho^2+2^{-2}(1-\rho^2)^2g_{\mathbb{S}^n}.
	$$
	Then $g_0$ is smooth on 
	$$\overline{\mathbb{H}}^{n+1}=\mathbb{H}^{n+1}\cup\partial\mathbb{H}^{n+1}.
	$$ 
	Now, along with  
	$$
	\begin{cases} \bar{\rho}=u^{\frac{p-1}{n}};\\
	\bar{g}=\bar{\rho}^2g=u^{\frac{2(p-1)}{n}}e^{2r}g_0,
	\end{cases}
	$$
	\eqref{{4.7}}'s continuity implies that $\bar{g}$ is H\"older continuous on $\overline{\mathbb{H}}^{n+1}$ and its boundary metric is
	$$\hat{g}=\bar{g}|_{{\mathbb{S}}^n}=\hat{u}^2g_0|_{\mathbb{S}^n}=2^{-2}{\hat{u}^2}g_{\mathbb{S}^n}.$$
	Accordingly, for any $s\in (0,1)$ we obtain the desired estimation
		\begin{equation*}
		\begin{aligned}
		{\rm cap}_p(\Omega)&=\int_{\{u=s\}}|\nabla u|^{p-1} d\sigma_g\\
		&=\int_{\{u=s\}}|\nabla\ln u|^{p-1}d\sigma_{\bar{g}}
		\\&=\lim\limits_{t\rightarrow 0}\int_{\{u=t\}}|\nabla\ln u|^{p-1}d\sigma_{\bar{g}}\\
		&\leq \left(\frac{n}{p-1}\right)^{p-1}\lim\limits_{t\rightarrow 0} \int_{\{u=t\}}d\sigma_{\bar{g}}
		\\&=\left(\frac{n}{p-1}\right)^{p-1} {\rm Vol}(\mathbb{S}^n,\hat{g})
	\\&=\left(\frac{n}{p-1}\right)^{p-1} \sigma_{\hat{g}}(\mathbb{S}^n),
		\end{aligned}
		\end{equation*}
		where we have used the known inequality (cf. \cite[Theorem 1.1]{2009LOCAL})
		$$|\nabla\ln u|\leq\frac{n}{p-1},
		$$
		whose equality is achievable when $\Omega$ is a hyperbolic geodesic ball and so is \eqref{{4.5}}'s equality. 
		
\end{proof}

\section{Sharp capacitary inequalities for smooth horospherical convex domains}\label{s5}

\par Another look at Theorem \ref{t41}(i) leads us to consider the concept that if $p\in (1,\infty)$ and $\Omega$ is a compact domain with smooth boundary $\partial\Omega$ in $\mathbb{H}^{n+1}$ then:
	\begin{itemize}
\item the $p$-capacity-radius of $\Omega$ is defined by
  $$ 
  {\rm Rcap}_p(\Omega)=R_0\ \ \text{with}\ \ {\rm cap}_p(\Omega)={\rm cap}_p\big(\overline{B}(R_0)\big);
  $$
  
\item the relative volume of $\Omega$ is defined by
$$
  {\rm RV}(\Omega)=\lim\limits_{R\rightarrow\infty} \left(\frac{{\upsilon}_g\big(\{x\in \mathbb{H}^{n+1}:\ d_g(x,\Omega)\leq R\}\big)}{\upsilon_g\big(B(R)\big)}\right);
$$
 
\item the relative $p$-capacity of $\Omega$ is defined by
 $$
   {\rm RC}_p(\Omega)=\lim\limits_{R\rightarrow\infty}\left(\frac{{\rm cap}_p\big(\{x\in \mathbb{H}^{n+1}:{\rm dist}_g(x,\Omega)\leq R\}\big)}{{\rm cap}_p\big(B(R)\big)}\right).
 $$
\end{itemize}

The fourth goal of this paper is to utilize some ideas developed within \cite{jin2024willmore, xiao2021flux} to obtain the following sharp capacitary estimation.

\begin{theorem}\label{t51} Let $p\in (1,\infty)$ and
  $\Omega$ be a compact subdomain of $\mathbb{H}^{n+1}$ with the smooth closed horospherical convex boundary $\partial\Omega$.
  
  \begin{itemize}
  	\item[\rm(i)] There is the inequality
  \begin{equation}\label{{5.1}}
 {\rm Rcap}_p(\Omega)\leq n^{-1}\ln {\rm RV}(\Omega)
 \end{equation}
 with equality iff $\Omega$ is a hyperbolic geodesic ball.
 
 \item[\rm(ii)] There is the inequality
 \begin{equation}
 \label{{5.2}}
{\rm Rvol}(\Omega)\le n^{-1}\ln{\rm RV}(\Omega)=n^{-1}\ln{\rm RC}_p(\Omega)
 \end{equation}
 with its inequality becoming an equality iff $\Omega$ is a hyperbolic geodesic ball.
\end{itemize}
\end{theorem}

\begin{proof} Below is a list of the essential properties of ${\rm RV}(\Omega)$ (cf. \cite{jin2024willmore} for more details):
\begin{itemize}
  \item 
  $$
  {\rm RV}(\Omega)=\lim\limits_{R\rightarrow\infty} \frac{{\upsilon}_g\big(\{x\in \mathbb{H}^{n+1}:\ d_g(x,\Omega)\leq R\}\big)}{\sigma_n\int_0^R \sinh^ns\,ds};
  $$
  \item 
  $$
  \Omega_1\subseteq \Omega_2\Rightarrow {\rm RV}(\Omega_1)\leq {\rm RV}(\Omega_2);
  $$
  \item $${\rm RV}(B_x(r))=e^{nr}\ \ \forall\ \ (x,r)\in \mathbb{H}^{n+1}\times(0,\infty);
  $$
  \item 
  $${\rm RV}(\Omega)\leq \int_{\partial\Omega\bigcap\{H\geq-n\}}\Bigg(1+\frac{H}{n}\Bigg)^n\,\frac{ d\sigma_g}{\sigma_n}\ \ \text{with}\ \ H=\text{the mean curvature};
  $$
  \item 
$${\rm RV}(\Omega)={\sigma_n}^{-1}{\rm Vol}(\mathbb S^n,\hat {g})\ \ \text{with}\ \ \hat{g}=2^{2}e^{-2d_g(\Omega,\cdot)}g_{\mathbb{H}^{n+1}}\big|_{\mathbb {S}^n}.
$$
\end{itemize}

(i) We begin with proving the assertion that if not only $\Omega\subseteq\mathbb{H}^{n+1}$ is a smooth bounded domain with horospherical
convex boundary $\partial \Omega$ but also
$$
\begin{cases}
f(s)=\sigma_n\sinh^ns=\sigma_g\big(\partial B(s)\big);\\
\Omega_s=\big\{x\in\mathbb{H}^{n+1}:\ d_g(x,\Omega)\leq s\big\},
\end{cases}
$$
then 
$$
[0,\infty)\ni s\mapsto f^{-1}\big(\sigma_g(\partial\Omega_s)\big)-s
$$ is a monotonically increasing function with
\begin{equation}
\label{{5.3}}
f^{-1}\big(\sigma_g(\partial \Omega)\big)\le f^{-1}\big(\sigma_g(\partial\Omega_s)\big)-s\le n^{-1}\ln {\rm RV}(\Omega)\ \ \forall\ \ s\in [0,\infty).
\end{equation}
Here $\Omega$ is said to be horospherical convex provided that $\Omega$ is contained in the enclosed ball of some horosphere through any given $x\in\partial\Omega$ - or equivalently - all the principal curvatures $\kappa_i$ of its boundary $\partial\Omega$ satisfy $$\min\{\kappa_1,\cdots,\kappa_n\}\geq 1.$$

The proof of this last assertion consists of three steps.

Firstly, if $\Omega$ is convex, then $\kappa_i\geq 0$, and hence the volume of level set $\partial\Omega_s$ can be accurately given by the Steiner
formula
\begin{equation*}
\sigma_g(\partial\Omega_s)= \sum_{k=0}^{n}\sinh^{k} s\cosh^{n-k} s \int_{\partial\Omega}{{n}\choose{k}} p_k\, d\sigma_g,
\end{equation*}
where $$p_k=\frac{1}{{{n}\choose{k}}}\sum\limits_{i_1<i_2<\cdots<i_k}\kappa_{i_1}\kappa_{i_2}\cdots\kappa_{i_k}$$ 
is the normalized $k-$th order
mean curvature. Moreover, when $\Omega$ is horospherically convex, by the hyperbolic Alexandrov-Fenchel inequality (cf. \cite{2018Hyperbolic, 2013Hyperbolic, li2014geometric, wang2014isoperimetric}) we have
\begin{equation*}
\int_{\partial\Omega}p_k\, d\sigma_g\geq \sigma_n\left(\left(\frac{\sigma_g(\partial\Omega)}{\sigma_n}\right)^{\frac{2}{k}}+\left(\frac{\sigma_g(\partial\Omega)}{\sigma_n}\right)^{\frac{2(n-k)}{kn}}\right)^{\frac{k}{2}}\ \ \forall\ \ k\in\{1,2,\cdots,n\}
\end{equation*}
with equality iff $\partial\Omega$ is a hyperbolic geodesic sphere. Also, if  $$
r_0=f^{-1}\big(\sigma_g(\partial\Omega)\big)-i.e.- \sigma_g(\partial\Omega)=\sigma_n\sinh^nr_0,
$$ 
then
\begin{align}\label{{5.4}}
\sigma_g(\partial\Omega_s)&\geq \sigma_g(\partial\Omega)\cosh^ns+ \sum_{k=1}^{n}\frac{\sinh^{k} s\cosh^{n-k}s  {{n}\choose{k}}\sigma_n}{\left(\left(\frac{\sigma_g(\partial\Omega)}{\sigma_n}\right)^{\frac{2}{k}}+\left(\frac{\sigma_g(\partial\Omega)}{\sigma_n}\right)^{\frac{2(n-k)}{kn}}\right)^{-\frac{k}{2}}}
\\
&= \sigma_g(\partial\Omega)\cosh^ns+ \sum_{k=1}^{n}\frac{\sinh^{k} s\cosh^{n-k}s {{n}\choose{k}}
\sigma_g(\partial\Omega)}{\left(1+\left(\frac{\sigma_g(\partial\Omega)}{\sigma_n}\right)^{-\frac{2}{n}}\right)^{-\frac{k}{2}}}\notag\\
&= \sigma_g(\partial\Omega)\cosh^ns+ \sigma_g(\partial\Omega)\sum_{k=1}^{n} {{n}\choose{k}} \sinh^{k} s\cosh^{n-k}s \coth^k r_0\notag\\
&=\sigma_g(\partial\Omega)\sum_{k=0}^{n}{{n}\choose{k}} (\sinh s \coth r_0)^k\cosh^{n-k}s\notag\\
&=\sigma_g(\partial\Omega)\left(\frac{\sinh^n(s+r_0)}{\sinh^n r_0}\right)\notag\\
&=f(s+r_0)\notag\\
&=f\big(s+f^{-1}(\sigma_g(\partial\Omega))\big),\notag
\end{align}
and hence 
$$
f^{-1}\big(\sigma_g(\partial\Omega_s)\big)\geq s+f^{-1}\big(\sigma_g(\partial\Omega)\big).
$$
Furthermore, for any $0<t<s,$ noticing the formula $$\Omega_{s}=(\Omega_{t})_{s-t},$$
we can similarly obtain 
\begin{equation*}
  f^{-1}\big(\sigma_g(\partial\Omega_{s})\big)\geq s-t+f^{-1}\big(\sigma_g(\partial\Omega_{t})\big)
\end{equation*}
provided that $\Omega_{t}$ is also horospherical convex. 

Secondly, let $h_{ij}(t)_{1\leq i,j\leq n}$ be the second fundamental
form of $\partial\Omega_t$ and $\lambda(t)$ be the smallest eigenvalue of $h_{ij}(t).$ Fix a $t_0>0,$ and assume that $V$ is the unit eigenvector
of $\lambda(t_0).$ Then 
$$V^Th(t)V-\lambda(t)\geq 0$$ 
whose  equality holds at $t_0.$ Hence
$$
\begin{cases}\lambda'(t_0)=V^Th'(t_0)V;\\
\lambda'(t_0)+\lambda^2(t_0)=V^Th'(t_0)V+V^Th^2(t_0)V=V^T\big(h'(t_0)+h^2(t_0)\big)V=1.
\end{cases}
$$
Here we have used the Riccati equation in the hyperbolic space:
$$h'_{ij}(t)+\sum\limits_{k=1}^n h_{ik}(t)h_{kj}(t)=-R_{titj}=1.
$$
Consequently, we get 
\begin{equation*}
\lambda'(t)+\lambda^2(t)=1\le \lambda(0),
\end{equation*}
thereby deriving 
$$
\lambda(t)\geq 1\ \ \forall\ \ t\geq0.
$$
Therefore, $\Omega_{t}$ is also horospherically convex for all $t\geq 0$. This in turn implies that 
$$
(0,\infty)\ni s\mapsto f^{-1}(\sigma_g(\partial\Omega_s))-s
$$ 
is a monotonically increasing function.

Thirdly, in order to obtain the upper bound of  $$f^{-1}(\sigma_g(\partial\Omega_s))-s,$$ we only need to make an estimation as $s$ tends to infinity. Recall that
\begin{equation*}
  {\rm RV}(\Omega)=\lim\limits_{s\rightarrow\infty}\left(\frac{\sigma_g(\partial\Omega_s)}{f(s)}\right).
\end{equation*}
So, if 
$$
\tau(s)=f^{-1}(\sigma_g(\partial\Omega_s))-s,
$$ 
then
$$
{\rm RV}(\Omega)=\lim\limits_{s\rightarrow\infty}
\left(\frac{f\big(\tau(s)+s\big)}{f(s)}\right)=\lim\limits_{s\rightarrow\infty} e^{n\tau(s)}
$$
whence validating \eqref{{5.3}}.

Now, we verify \eqref{{5.1}} via using the classical methods from \cite{polya1951isoperimetric}. As a matter of fact, we always have
 \begin{equation}\label{{5.5}}
   {\rm cap}_p(\Omega)\leq \left(\int_0^{\infty} \sigma_g(\partial\Omega_s)^{\frac{1}{1-p}}ds\right)^{1-p}.
 \end{equation}
 Since \eqref{{5.3}} reveals
 $$\sigma_g(\partial\Omega_s)\leq f\big(s+n^{-1}\ln {\rm RV}(\Omega)\big),
 $$ 
 this last inequality is placed into \eqref{{5.5}} to deduce
  \begin{equation}\label{{5.6}}
   {\rm cap}_p(\Omega)\leq \left(\int_{n^{-1}\ln {\rm RV}(\Omega)}^{\infty} f(s)^{\frac{1}{1-p}}ds\right)^{1-p}.
 \end{equation}
 Consequently,  \eqref{{5.6}}, plus the definition of ${\rm Rcap}_p(\Omega)$, derives \eqref{{5.1}} at once. 
 
 Clearly, if $\Omega$ is a hyperbolic geodesic ball, then \eqref{{5.1}}'s equality is true. 
 Conversely, if the equality of \eqref{{5.1}} holds, then 
 $$
 \sigma_g(\partial\Omega_s)= f\big(s+n^{-1}\ln {\rm RV}(\Omega)\big)\ \ \forall\ \ s\in (0,\infty),
 $$
 and hence 
 $$
 f^{-1}(\sigma_g(\partial\Omega_s))-s
 $$ 
 is a constant function since it is a monotonically increasing function. This last constant reveals that each inequality of \eqref{{5.4}} becomes an equality. So, $\Omega$ must be a geodesic ball.

(ii) On one hand, upon letting $p\rightarrow 1$ in \eqref{{5.5}}, we get that
$$
 {\rm cap}_1(\Omega)\leq \lim\limits_{p\rightarrow 1}\left(\int_0^{\infty} \sigma_g(\partial\Omega_s)^{\frac{1}{1-p}}ds\right)^{1-p}
=\min\limits_{s\geq 0} \sigma_g(\partial\Omega_s)
=  \sigma_g(\partial\Omega)$$
Meanwhile, Theorem \ref{t41}(i) implies 
$${\rm cap}_{1}(\Omega)\geq {\rm cap}_{1}\Big(\overline{B}\big({\rm Rvol}(\Omega)\big)\Big)=\sigma_g\Big(\partial B\big({\rm Rvol}(\Omega)\big)\Big)=f\big({\rm Rvol}(\Omega)\big).
$$
Thus, the last two inequalities derive
$$
{\rm Rvol}(\Omega)\leq f^{-1}\big({\rm cap}_{1}(\Omega)\big)={\rm Rcap}_{1}(\Omega)\leq f^{-1}(\sigma_g(\partial\Omega))\leq n^{-1}\ln{\rm RV}(\Omega),
$$ 
thereby reaching \eqref{{5.2}}'s inequality which becomes an equality iff $\Omega$ is a hyperbolic geodesic ball due to the last estimation.

On the other hand, with the help of \cite{jin2022relative}'s argument for \eqref{{5.2}} 
whenever $\Omega$ {is either a point or a geodesic ball in any asymptotically hyperbolic Einstein manifold}, we prove that 
if $\Omega$ is just a bounded domain with smooth boundary $\partial\Omega$ in $\mathbb H^{n+1}$, then the second equation of \eqref{{5.2}} is still valid.
 
For convenience, we once again use the symbol 
$$
\Omega_R=\big\{x\in \mathbb{H}^{n+1}:\ d_g(x,\Omega)\leq R\big\}.
$$
By the flux method, for any $R>0$ we get
  \begin{align}\label{{5.7}}
  \limsup\limits_{R\rightarrow\infty}\Bigg(\frac{{\rm cap}_p(\Omega_R)}{{\rm cap}_p(B(R))}\Bigg)&
  \leq \limsup\limits_{R\rightarrow\infty}\left(\frac{\left(\int_R^{\infty} \sigma_g(\partial\Omega_s)^{\frac{1}{1-p}}ds\right)^{1-p}}{\left(\int_R^{\infty} f(s)^{\frac{1}{1-p}}ds\right)^{1-p}}\right)\\
  &=\left(\lim\limits_{R\rightarrow\infty}\Bigg(\frac{\int_R^{\infty} \sigma_g(\partial\Omega_s)^{\frac{1}{1-p}}ds}{\int_R^{\infty} f(s)^{\frac{1}{1-p}}ds}\Bigg)\right)^{1-p}\notag
  \\
  &=\left(\lim\limits_{R\rightarrow\infty}\Bigg(\frac{\sigma_g(\partial\Omega_R)^{\frac{1}{1-p}}}{f(R)^{\frac{1}{1-p}}}\Bigg)\right)^{1-p}\notag
  \\ &=\lim\limits_{R\rightarrow\infty}\left(\frac{\sigma_g(\partial\Omega_R)}{f(R)}\right)\notag \\
  &={\rm RV}(\Omega).\notag
  \end{align}
 
 In order to establish the reversed estimation of \eqref{{5.7}}, let $\tilde{R}={\rm Rvol}\big(\Omega)_R\big)$. Then Theorem \ref{t41}(i) gives
\begin{equation}\label{{5.8}}
\frac{{\rm cap}_p(\Omega_R)}{{\rm cap}_p\big(\overline{B}({R})\big)}\geq\frac{{\rm cap}_p\big(\overline{B}(\tilde{R})\big)}{{\rm cap}_p\big(\overline{B}({R})\big)}
\end{equation}
Note that \cite[(6.3)]{jin2022relative} implies
$${\rm cap}_p\big(\overline{B}({R})\big)={2^{-n}}\left(\frac{n}{p-1}\right)^{p-1}\sigma_n e^{nR}+{\smallO}(e^{nR})\ \ \text{as}\ \ R\rightarrow\infty.
$$
So
$${\rm cap}_p\big(\overline{B}({R})\big)=n\sigma_n\left(\frac{n}{p-1}\right)^{p-1}\upsilon_g\big(B(R)\big)+{\smallO}\Big(\upsilon_g\big(B(R)\big)\Big) \ \ \text{as}\ \  R\rightarrow\infty.
$$
Furthermore, by \eqref{{5.8}} we achieve
\begin{align}\label{{5.9}}
\liminf\limits_{R\rightarrow\infty}\left(\frac{{\rm cap}_p(\Omega_R)}{{\rm cap}_p\big(\overline{B}({R})\big)}\right)
 & \geq\liminf\limits_{R\rightarrow\infty}\left(\frac{n\sigma_n\big(\frac{n}{p-1}\big)^{p-1}\upsilon_g\big(B(\tilde{R})\big)+{\smallO}\Big(\upsilon_g\big(B(\tilde{R})\big)\Big)
}{n\sigma_n\big(\frac{n}{p-1}\big)^{p-1}\upsilon_g\big(B(R)\big)+{\smallO}\Big(\upsilon_g\big(B(R)\big)\Big)}\right)
\\ &=\liminf\limits_{R\rightarrow\infty}\left(\frac{\upsilon_g\big(B(\tilde{R})\big)+{\smallO}\Big(\upsilon_g\big(B(\tilde{R})\big)\Big)}{\upsilon_g\big(B({R})\big)+{\smallO}\Big(\upsilon_g\big(B({R})\big)\Big)}\right)\notag
\\ &=\liminf\limits_{R\rightarrow\infty}\left(\frac{\upsilon_g(\Omega_R)+o\big(\upsilon_g(\Omega_R)\big)}{\upsilon_g\big(B({R})\big)+{\smallO}\Big(\upsilon_g\big(B({R})\big)\Big)}\right)\notag
\\ &=\lim\limits_{R\rightarrow\infty}\left(\frac{\upsilon_g(\Omega_R)}{\upsilon_g\big(B({R})\big)}\right)\notag \\
&={\rm RV}(\Omega).\notag 
\end{align}

Accordingly, a combination of \eqref{{5.7}}\&\eqref{{5.9}} derives \eqref{{5.2}}'s second identification.
 \end{proof}

\end{document}